\documentclass[12pt, final]{amsproc}
\usepackage[cp850]{inputenc}
\usepackage{amssymb}
\usepackage[all]{xy}
\usepackage[mathscr]{eucal}
\usepackage{amscd}
\usepackage{amsmath}
\usepackage{setspace}
\usepackage{tikz}
\newcommand{\R}{\mathbb R}
\newcommand{\e}{\varepsilon}
\newcommand{\N}{\mathbb N}

\def\Lip{\operatorname{Lip}}
\newcommand{\U}{\mathscr{U}}
\newcommand{\PO}{\mathrm{PO}}

\theoremstyle{plain}
\newtheorem{theorem}{Theorem}
\newtheorem{teor}{Theorem}
\newtheorem{prop}{Proposition}
\newtheorem{lemma}{Lemma}
\newtheorem{cor}{Corollary}
\newtheorem{corollary}{Corollary}

\theoremstyle{remark}

\newcommand{\adef}{\begin{defn}}
\newcommand{\zdef}{\end{defn}}
\newtheorem{defn}[theorem]{Definition}

\def\sep{\operatorname{sep}}
\title{The separable Jung constant in Banach spaces}

\author{Jes\'us M. F. Castillo}
\address{Instituto de Matemßtica de la Universidad de Extremadura,
Avenida de Elvas, 06071-Badajoz, Spain}
\email{castillo@unex.es}

\author{Pier Luigi Papini}
\address{via Martucci 19, 40136 Bologna, Italia}
\email{pierluigi.papini@unibo.it}

\subjclass[2010]{46B04, 46B20, 46B26, 46M18 }
\thanks{The research of the first author was supported in part by Project IB16056 de la Junta de Extremadura and MINCIN Project MTM2016-76958-C2.}
\begin{document}

\thanks{The authors thank the referee for his/her splendid job and many poignant comments that helped enormously to improve the quality of the paper.}

\keywords{$1$-separable injectivity; Jung constant; Kottman constant}
\maketitle

\begin{abstract}  This paper contains a study of the separable version $J_s(\cdot)$ of the classical Jung constant. We first establish, following Davis \cite{davis}, that a Banach space $X$ is $1$-separably injective if and only if $J_s(X)=1$. This characterization is then used for the understanding of new $1$-separably injective spaces. The last section establishes the inequality $\frac{1}{2}K(Y)J_s(X)\leq e_1^s(Y,X)$ connecting the separable Jung constant, Kottman's constant and the separable-one-point extension constant for Lipschitz maps, which is then used to derive an improved version of Kalton's inequality $K(X,c_0)\leq e(X,c_0)$ and a new characterization of $1$-separable injectivity.\end{abstract}

\section{The Jung constants}

Given a bounded subset $A\subset X$ we define the diameter of $A$ as $\delta(A) = \sup \{ \|a-b\|: a, b\in A\}$, while the radius of $A$ is defined by $r(A) = \inf_{b\in X} \sup_{a\in A}  \|a-b\|$. The Jung constant \cite{jung} of $X$ is defined as
$$J(X) = \sup \frac{2r(A)}{\delta(A)} $$
where the supremum is taken over all closed bounded sets $A$ with $\delta(A)>0$. It was shown by Davis in \cite{davis} and by Franchetti in \cite{fran2} that
a Banach space with Jung constant 1 is $1$-injective. It follows then from the work of Lindenstrauss \cite{lindmemo} that a Banach space is $1$-injective if and only if $J(X)=1$. Recall that a  Banach space $X$ is $\lambda$-injective if for every Banach space $F$ and every subspace $E$ of $F$ every operator $t: E\to X$ has an extension $T: F\to X$ with $\|T\|\leq \lambda\|t\| $.\medskip

Two important variations of this notion are: $\lambda$-separable injectivity, when the property above holds when $F$ is separable; and universal $\lambda$-separable injectivity \cite{accgm1, accgmLN}, when the preceding property holds when $E$ is separable. Other cardinal variations of this notion were studied in \cite{accgm4} and \cite{accgmLN}. Accordingly, given an uncountable cardinal $\aleph$, a Banach space $X$ is said to be $(\lambda, \aleph)$-injective if the preceding condition holds for Banach spaces $F$ with density character less than $\aleph$. As it is remarked in \cite[Def. 5.1]{accgmLN}, ``the resulting name for separable injectivity turns out to be $\aleph_1$-injectivity, not $\aleph_0$-injectivity, as one would have expected. Nevertheless, we have followed the uses of set theory where properties labeled by a cardinal always indicate that
something happens for sets whose cardinality is strictly less than that cardinal".\medskip

Turning back to the connections between injectivity properties and the Jung constant, recall that a Banach space is $1$-injective if and only if
every family of mutually intersecting balls has nonempty intersection and observe that $1$-injective spaces are $1$-complemented in
any larger superspace; in particular, in larger one codimensional superspaces ($F$ is a one codimensional superspace of $E$ if $\dim F/E=1$). This is the core of the proof that $1$-injectivity can be characterized by properties of intersection of balls \emph{of equal radius}. Namely, a Banach space is $1$-injective if and only if every family of mutually intersecting balls of radius one admits nonempty intersection; i.e., $J(X)=1$. Such is the content of the proof of Franchetti \cite{fran2}: indeed, it is based on the estimate $1\leq J(X)\leq \lambda_1(X)\leq g(J(X))$ where $g:[1,2]\to [1,2]$ is a certain function such that $g(1)=1$.\medskip

What occurs under cardinal restrictions? For instance, are the properties \emph{every countable family of mutually intersecting balls has nonempty intersection} and \emph{every family of mutually intersecting balls of radius one has nonempty intersection} equivalent?
The former of those properties corresponds \cite{accgmLN} to $1$-separable injectivity. To treat the latter we introduce the \emph{separable} Jung constant $J_s(\cdot)$ defined as $J(\cdot)$ but considering only separable bounded sets (there is an implicit approach to this notion in \cite{bayana}).

\adef Given an uncountable cardinal $\aleph$ we define the Jung constant
$$J_\aleph(X) = \sup \frac{2r(A)}{\delta(A)} $$
where the supremum is taken over all closed bounded sets $A$ with $\delta(A)>0$ and density character strictly less than $\aleph$.
\zdef

Thus, $J_s(X) = J_{\aleph_1}(X)$; and the question is then whether $J_s(X)=1$ characterizes $1$-separable injectivity. The straight way to derive this result from the injectivity result cannot work because, suprising as it may seems, $1$-separably injective spaces are not necessarily $1$-complemented in one-codimensional superspaces (see the proof below) and this is what makes difficult to find a way to adapt Franchetti's proof for separable injectivity. Davis proof \cite{davis} again relies on the fact of proving that spaces $X$ with $J(X)=1$ are $1$ complemented in any superspace $F$ so that $\mathrm{dim} F/X=1$ and therefore it cannot work under cardinal restrictions. To make sound all these assertions, we need a couple of elements of homological Banach space theory. The first one is the notion of  exact sequence $0\to Y \to X \to Z\to 0$, which is a diagram formed by Banach spaces and linear continuous operators with the property that the kernel of each arrow coincides with the image of the preceding. Thanks to the open mapping theorem this exactly means that $Y$ is (isomorphic to) a subspace of $X$ and $Z$ is (isomorphic to) the corresponding quotient $X/Y$. An exact sequence is said to split if the subspace $Y$ is complemented in $X$; i.e., there is a linear continuous projection $P: X\to Y$. The second thing is the  push-out construction (see complete details in \cite{accgmLN}), which operatively defined means that whenever one has an exact sequence  $0\to Y \to X \to Z\to 0$ and an operator $\tau: Y\to Y'$ there is a commutative diagram

\begin{equation}\label{podiagram}\begin{CD}
0@>>> Y @>>> X @>>> Z @>>>0\\
&&@V{\tau}VV @VV{T}V @|\\
0@>>> Y' @>>> \PO @>>>Z@>>>0\end{CD}\end{equation}

The push-out space $\PO$ is the quotient $(Y'\oplus_1 X)/\Delta(Y)$ where $\Delta: Y\to Y'\oplus_1 X$ is the isomorphic embedding $\Delta(y)= (\tau(y), - y)$. What is required to know to follow the arguments in this paper is:

\begin{enumerate}
\item The lower sequence in diagram (\ref{podiagram}) is exact.
\item Although, as we said before, the operator $Y\to X$ in an exact sequence as above only has to be an isomorphic embedding, in the exact sequences we will consider it is an isomeric embedding. When the inclusion $Y\to X$ (resp $\tau$) is an into isometry, the same occurs to the inclusion $Y'\to \PO$ (resp. to $T$).
\item $T$ is always a contractive operator.
\item The lower sequence in a push-out diagram as above splits if and only if $\tau$ admits a linear continuous extension $\widehat \tau: X\to Y'$. \end{enumerate}
All these facts can be encountered with proofs in \cite{accgmLN}. We prove the announced result. We are sure it is well known; but since we could not find an explicit reference, we include its proof for the sake of completeness.

\begin{lemma} A Banach space $X$ is $1$-injective if and only if it is $1$-complemented in every one codimensional superspace.
\end{lemma}
\begin{proof} One implication, as it has already been mentioned, is clear. Assume that $X$ is $1$-complemented in every one codimensional superspace and let us show that every operator $\tau: E\to X$ admits an equal norm extension to any one codimensional superspace $F$ of $E$. Form the push-out diagram
$$\begin{CD}\label{podiagram}
0@>>> E @>>> F @>>> \R @>>>0\\
&&@V{\tau}VV @VV{T}V @|\\
0@>>> X @>>> \PO @>>>\R@>>>0\end{CD}$$
which clearly means that $X$ is of codimension 1 in $\PO$. Let $P: \PO\to X$ be a norm one projection. The composition $PT$ yields the equal norm extension of $\tau$.\end{proof}

Since, on the other hand, exist $1$-separably injective spaces that are not $1$-injective \cite{accgm1,accgmLN}, one gets that $1$-separably injective are not necessarily $1$-complemented in one codimensional superspaces. Let us show that Davis proof can be nevertheless modified to cover the separable (as well as other cardinal restrictions) case. Precisely:

\begin{teor}\label{jung} Let $\aleph$ be an infinite cardinal. A Banach space $X$ is $(1,\aleph)$-injective if and only if $J_\aleph(X)=1$. In particular, $X$ is $1$-separably injective if and only if $J_s(X)=1$.
\end{teor}
\begin{proof} We will make the proof for $\aleph=\aleph_1$ that corresponds to separable injectivity and indicate the modifications to be made for other cardinals. What we will actually show is that $J_s(X)=1$ implies that every norm one operator $\tau: E\to X$ with $E$ separable (in the general case will be $E$ having density character $<\aleph$) admits, for every $\varepsilon>0$, an extension $\tau_\varepsilon: F\to X$ with $\|\tau_\varepsilon\|\leq 1+\varepsilon$ whenever $F$ is a superspace of $E$ with $\dim F/E=1$. So we set $\R= F/E$ and denote by $f: F \to \R$ the quotient map. Form the commutative diagram

$$\begin{CD}
0@>>> E @>>> F @>f>> \R @>>>0\\
&&@V{\tau}VV @VV{T}V @|\\
0@>>> \overline{ \tau(E)} @>>>G @>g>>\R@>>>0\\
&&@V{\imath}VV @VV{\jmath}V @| \\
0@>>>X @>>> H @>h>> \R @>>>0 \end{CD}$$\\

in which $G$ and $H$ are the corresponding push-out spaces and we have labeled $g, h$ the respective quotient maps. The operator $T$ has norm one and $\jmath$ is an isometric embedding, which clearly shows that it is enough to obtain an extension $\imath_\varepsilon: G \to X$ of $\imath$ with norm at most $1+\varepsilon$. To this purpose, observe that for each $0<\varepsilon<1$ the metric space $h^{-1}(\varepsilon)$ is isometric to $X$. Thus, $B(0,1)\cap h^{-1}(\varepsilon)\cap \jmath(G)$ is a (separable) set of diameter at most $2$ of $X$ must have radius $1$; which means that there exists $z_\varepsilon$ in $h^{-1}(\varepsilon)$ so that
$B(0,1)\cap h^{-1}(\varepsilon) \cap \jmath(G) \subset B(z_\varepsilon, 1)$. We define an extension $\imath_\varepsilon: G\to X$ of the canonical inclusion $\imath$ as
$$\imath_\varepsilon(x) = \jmath(x)  - \frac{g(x)}{\varepsilon}z_\varepsilon.$$

\noindent \textbf{Claim 1.  If $\|x\|\leq 1$ and $g(x) = \varepsilon$ then $\|\imath_\varepsilon(x)\|\leq 1$}.\medskip
\begin{proof}[Proof of the Claim] Pick $\|x\|\leq 1$. If $g(x)=\varepsilon$ then $x\in G\cap B(0,1)\cap g^{-1}(\varepsilon) $ and thus
$\jmath x\in \jmath(G) \cap B(0,1)\cap h^{-1}(\varepsilon) $ so that
$\|\imath_\varepsilon(x)\| = \|\jmath(x) - z_\varepsilon\| \leq 1$ by the choice of $z_\varepsilon$.\end{proof}

Now, if $\beta>\varepsilon$ then define the map $\phi_\beta:  h^{-1}(\varepsilon) \to  h^{-1}(\beta)$ given by
$$\phi_\beta(y) = y +  \frac{\beta - \varepsilon}{\varepsilon}z_\varepsilon.$$

Let us check that if $\|x\|=1$ then $\|\phi_\beta(\jmath  (x))\|\geq 1$. When $x\in g^{-1}(\varepsilon)$ one has $\imath_\varepsilon(x) = \jmath(x) - z_\varepsilon$ and thus it turns out that
\begin{eqnarray*}
\jmath(x) &=& \frac{\beta - \varepsilon}{\beta} (\jmath(x) - z_\varepsilon) + \left(1 - \frac{\beta - \varepsilon}{\beta}\right)\left(\jmath(x) +  \frac{\beta - \varepsilon}{\varepsilon}z_\varepsilon \right)\\
&=& \frac{\beta - \varepsilon}{\beta} \;\imath_\varepsilon(x)+ \left(1 - \frac{\beta - \varepsilon}{\beta}\right)\phi_\beta(\jmath(x)).
\end{eqnarray*}
Therefore, if $\|x\|= 1$ and $g(x)=\varepsilon$, since $\|\imath_\varepsilon(x)\|\leq 1 = \|\jmath (x)\|$, one then has $\|\phi_\beta\jmath(x)\|\geq 1$ as we claimed.\medskip

\noindent \textbf{Claim 2. $ \jmath\left( B(0,1)\cap g^{-1}(\beta)\right) \subset \phi_\beta \jmath \left (B(0,1)\cap g^{-1}(\varepsilon) \right)$.}
\begin{proof}[Proof of the Claim] To simplify the exposition of this part let us call $B_\varepsilon = B(0,1)\cap g^{-1}(\varepsilon)$. The statement to prove is then $ \jmath\left( B_\beta \right ) \subset \phi_\beta \jmath\left( B_\varepsilon \right) $. Consider the convex set

$$\mathcal C = \bigcup_{z \in B_\varepsilon} [\imath_\varepsilon(z), \phi_1 \jmath(z)] = \jmath(B_\varepsilon) + [- z_\varepsilon, \frac{1 - \e}{\e}z_\e].$$
Moreover
$$ \phi_\beta \jmath\left( B_\varepsilon \right) = \mathcal C \cap h^{-1}(\beta)$$

since clearly $\phi_\beta \jmath\left( B_\varepsilon \right)\subset h^{-1}(\beta)$ and also picking $z\in B_\varepsilon$

$$\phi_\beta\jmath(z) = (1-\beta)\left(\jmath z - \frac{g(z)}{\varepsilon}z_\varepsilon\right) + \beta\left(\jmath z + \frac{1-\varepsilon}{\varepsilon}z_\varepsilon\right)\in [\imath_\varepsilon (z), \phi_1\jmath(z)]\subset \mathcal C.$$
This proves the inclusion $\subset$. Now we are going to prove the claim itself. Pick a point $\jmath(x) \in \jmath\left( B_\beta \right ) $ and assume it is not in $\phi_\beta \jmath\left( B_\varepsilon \right) =\mathcal C \cap \jmath g^{-1}(\beta)$. This means that if one picks a point $y\in B_\varepsilon$ with $\|y\|<1$ ---which must be in the interior of $\mathcal C$--- then the segment $[\jmath (y), \jmath (x)]$ must intersect the boundary of $\mathcal C$, which is the union of the three sets
\begin{eqnarray*}
\mathcal C_1 &=& \{\imath_\varepsilon(z) : z\in B_\varepsilon\}\\
\mathcal C_2 &=& \{\phi_1 \jmath (z) : z\in B_\varepsilon\}\\
\mathcal C_3 &=& \bigcup_{\|z\|=1; z\in B_\e} [\imath_\varepsilon(z), \phi_1 \jmath (z)]\\
\end{eqnarray*}
Since $h\jmath(y)=g(y)=\varepsilon$ and $h\jmath(x) = g(x)=\beta$ all points $u\neq \jmath(y)$ in $[\jmath (y), \jmath (x)]$ have $\varepsilon < h(u)<\beta$; so $[\jmath (y), \jmath (x)] \cap \mathcal C_1$ is impossible since $h\imath_\varepsilon(z)=0$ for every $z\in \mathcal C_1$. The intersection
$[\jmath (y), \jmath (x)]\cap\mathcal C_2$ must also be empty since $h \phi_1 \jmath(z) = g \phi_1(z) =\beta$ for every $z\in \mathcal C_2$. The only remaining possibility is the existence of some $u \in [\jmath (y), \jmath (x)] \cap \mathcal C_3$. And this is impossible because it forces $u$ to have $\|u\|\geq 1$: indeed, recall that, as we showed above, $\jmath(z)$ is the norm one point in the interval $[\imath_\varepsilon(z), \phi_1 \jmath(z)]$, and $h\jmath(z)=\varepsilon$.
Assume $u\in (\jmath(y), \jmath(x))\cap \mathcal C_3$. Since $\|\jmath(y)\|\leq 1$ and $\|\jmath(x)\|\leq 1$, necessarily $\|u\|<1$. Further, $h(u)\in (h\jmath(y), h\jmath(x))=(\e, \beta)$. On the other hand, $u\in (\imath_\e(z), \phi_1\jmath(z))$ for some $z\in B_\e$ with $\|z\|=1$. Since
 $h\jmath(z)=\e$ and $h\phi_1 \jmath(z)=\beta$, necessarily  $u\in (\jmath(z), \phi_1\jmath(z))$. Finally, $\|\jmath(z)\|=1, \|\imath_\e(z)\|\leq 1$ by Claim 1 and $\|\phi_1\jmath(z)\| \geq 1$ by the claim between the Claims 1 and 2. Thus $\|u\|\geq 1$, a contradiction.\end{proof}

Now, from the containment in Claim 2 we immediately obtain
$$\imath_\varepsilon \left( B(0,1)\cap g^{-1}(\beta) \right) \subset \phi_\beta \jmath\left( B(0,1)\cap g^{-1}(\varepsilon) \right)- \frac{\beta}{\e}z_\e\subset B(0,1)$$
since $\imath_\varepsilon( \jmath(x) + \frac{\beta-\varepsilon}{\varepsilon}z_\varepsilon) =\imath_\e \jmath (x)= \jmath (x)  - z_\e$ (for the first inclusion); and for every $x\in B(0,1)\cap g^{-1}(\varepsilon)$ one has
$\phi_\beta \jmath(x) - \frac{\beta}{\e}z_\e = \jmath(x) + \frac{\beta - \e}{\e}z_\e - \frac{\beta}{\e}z_\e =  \jmath(x) - z_\e \in B(0,1)$
(for the second). Therefore, if $\|x\|\leq 1$ and $g(x)=\beta>\varepsilon$, namely, if $x \in B(0,1) \cap g^{-1}(\beta)$ then we have just shown $\|\imath_\varepsilon(x)\|\leq 1$. This, in combination with Claim 1 proves \medskip

\noindent \textbf{Claim 3. If $\|x\|=1$ and $g(x) \geq \varepsilon$ then $\|\imath_\varepsilon(x)\|\leq 1$}.\medskip

We are ready to conclude that $\|\imath_\varepsilon\|\leq 1+2\varepsilon$ can be obtained, and the only remaining case is when $\|x\|=1$ and $g(x)<\varepsilon$.
An immediate consequence of the preceding arguments is that $\|z_\varepsilon\| \leq  2\varepsilon/\beta<2$: pick $\|v\|=1$ with $\varepsilon < g(v)=\beta<1$. Since
$\|\imath_\varepsilon( v) \|= \|\jmath (v) - \frac{\beta}{\varepsilon}z_\varepsilon\| \leq 1$ it follows that $\|z_\varepsilon\| \leq  2\frac{\varepsilon}{\beta}<2$. And from that
we obtain $\|\phi_\beta z_\varepsilon\| = \|\frac{\beta}{\varepsilon}z_\varepsilon\| <2$. The proof concludes now: since
$$\imath_\varepsilon (x) = \jmath( x) - \frac{g(x)}{\varepsilon}z_\varepsilon = \jmath (x ) - \frac{g(x)}{\beta }\phi_\beta z_\varepsilon $$
it follows that
$$\|\imath_\varepsilon (x) \|\leq  1 + \frac{2}{\beta}|g(x)| \leq  1 + \frac{2\varepsilon}{\beta} \leq 1 + 3\e.$$
just picking $\beta>2/3$.\medskip

We have proved that there exist, for every $\e>0$, $(1+3\varepsilon)$-extensions to one-codimensional superspaces. We need to show now that there exist equal norm extensions to one-codimensional superspaces. The point is the equivalence between extension properties and properties of intersections of balls. Nachbin \cite{nachbin} (see also \cite[Theorem 2.1]{zipphandbook}) proved that a Banach space $X$ is $1$-injective if and only if every family of mutually intersecting balls has nonempty intersection. It can be attributed to Lindenstrauss that a Banach space is $1$-separably injective if and only if every \emph{countable} family of mutually intersecting balls has nonempty intersection: see the comments in \cite[Proof of the theorem]{lindCK}. As well as it can be attributed to him, in general (see \cite{accgm4} and \cite[5.2 and 5.5.4]{accgmLN}), that a Banach space is $(1,\aleph)$-injective if and only if every family of size $< \aleph $ of mutually intersecting balls has nonempty intersection. A combination of those ideas with the arguments of Lindenstrauss in \cite[Theorem 6.10]{lindmemo} yields

\begin{lemma} If every sequence $B(x_n, r_n)$ of mutually intersecting balls is such that $\bigcap_n B(x_n, r_n+\varepsilon)\neq \emptyset$ for every $\varepsilon>0$ then any countable family $B(x_n, r_n)$ of mutually intersecting balls has $\bigcap_n B(x_n, r_n)\neq \emptyset$. The result remains valid for families of size $< \aleph$.\end{lemma}

Let us remark that also Davis in \cite[p.316]{davis} formulates the compactness argument for radius $1$ balls in the forms: \emph{If every sequence $B(x_n, 1)$ of mutually intersecting balls is such that $\bigcap_n B(x_n, 1+\varepsilon)\neq \emptyset$ for every $\varepsilon>0$ then also every sequence $B(x_n, 1)$ of mutually intersecting balls is such that $\bigcap_n B(x_n, 1)\neq \emptyset$}.\medskip

This means that a Banach space $X$ is $1$-separably injective if and only if, for every $\varepsilon>0$, every separable subspace $S\subset \ell_\infty$ and every $x\in \ell_\infty$, every norm one operator $\tau: S\to X$ admits an extension $\tau_\varepsilon: S + [x]\to X$ with $\|\tau_\varepsilon\|\leq 1+\varepsilon$. The $(1,\aleph)$-injective) version is easy to reformulate too. This concludes the proof.\end{proof}

The geometrical conclusion of the previous proof is:

\begin{cor} Let $\aleph$ be a cardinal and let $X$ be a Banach space. Every family of size $<\aleph$ of mutually intersecting balls has nonempty intersection if and only if every family of size $<\aleph$ of mutually intersecting balls of radius $1$ has nonempty intersection.\end{cor}
\begin{proof} One implication is obvious. As for the other, let us again simplify the exposition working with $\aleph=\aleph_1$. If every countable family of mutually intersecting balls of radius $1$ has nonempty intersection then $J_s(X)=1$: indeed, let $A$ be a separable set with diameter 2 and let $\{x_n\}_n$ be a countable dense subset of $A$. The intersection of each two balls $B(x_i, 1)\cap B(x_j, 1)$ is nonempty
and thus there is some $p\in \cap_n B(x_n, 1)$. This means that $A\subset B(p, 1+\varepsilon)$ and the radius of $A$ is at most $1+\varepsilon$, which implies $J_s(X)=1$. Thus $X$ is $1$-separably injective by Theorem \ref{jung}, hence every countable family of mutually intersecting balls has nonempty intersection.\end{proof}

Thus, the difference between $1$-separable (resp. $(1, \aleph)$) injectivity and having separable Jung (resp. $J_\aleph$) constant 1 is whether intersection properties of sequences (resp. families of size $<\aleph$) of radius $1$ balls pass to arbitrary sequences (resp. families of size $<\aleph$) of balls, as it is the case.

\section{Uses and Applications}

\subsection{Bounded sets without center} Thus, in $1$-separably injective spaces that are not $1$-injective there are closed bounded sets
with diameter 2 but radius $>1$ and these cannot be separable. The three main examples in the literature (see \cite{accgmLN}) of $1$-separably injective spaces that are not $1$-injective are:
\begin{enumerate}
\item Given an uncountable set $\Gamma$ then the space $\ell_\infty^c(\Gamma)$ of bounded functions with countable support.
\item The space $\ell_\infty/c_0$.
\item Ultrapowers of Lindenstrauss spaces with respect to countably incomplete ultrafilters on $\N$.
\end{enumerate}

It is worth to carefully exam why such exotic bounded sets exist in those spaces.
\begin{itemize}
\item Fix as $\Gamma$ a set of size the continuum $\mathfrak c$. The set $V$ of elements of $\ell_\infty^c(\Gamma)$ taking values in $[0,1]$ has diameter $1$ but no center for radius $1/2$ balls since such center should be the constant function, which is not in $\ell_\infty^c(\Gamma)$. On the other hand splitting $\Gamma = \Gamma_1 \cup \Gamma_2$ with both $\Gamma_i$ uncountable then the set $A_1 \cup - A_2$ with $A_i = \{f: \mathrm{supp} f \subset \Gamma_i\} $ has diameter 1 and radius 1 (this example is taken from \cite{amir}). Hence $J(\ell_\infty^c(\Gamma) )=2$.

\item An example of a diameter $1$ set with radius $1$ in $\ell_\infty/c_0$ has been kindly provided to us by Manuel Gonz\'alez: Suppose that $\{r_i : i \in I\}$ is the set of all branches in the dyadic tree $\Delta$, so that $|I|=\mathfrak c$. We enumerate the nodes of $\Delta$ in the usual way, so that we can assume $\Delta = \N$. Let $1_i$ be the characteristic function of $r_i$, so that we can assume $1_i \in \ell_\infty$. As it is well-known, the images $g_i$ of $1_i$ in $\ell_\infty/c_0$ generate a subspace isometric to $c_0(I)$ because the intersection of two different branches is finite. As a consequence, for every $\varepsilon \in \{-1,1\}^I$ the set $A_\varepsilon = \{\varepsilon(i) \cdot g_i : i \in I\}$ has diameter 1. Now, for each $r<1$ there is a choice of signs $\varepsilon \in \{-1,1\}^I$ so that $A_\varepsilon$ cannot be contained in a ball $B(h_\varepsilon,r)$  since otherwise
\begin{eqnarray*}2 &=&  \|\varepsilon(i)\cdot g_i  - \eta(i)\cdot g_i \|\\ &\leq& \|\varepsilon_i\cdot g_i - h_\varepsilon\|+ \|h_\varepsilon - h_\eta\| + \|h_\eta- \eta_i\cdot g_i\|\\
&\leq& 2r + \|h_\varepsilon - h_\eta\|\end{eqnarray*}
yields $\|h_\varepsilon - h_\eta\| \geq 2(1-r)$. Thus, $\{h_\varepsilon : \varepsilon \in \{-1,1\}^I\}$ is a $2(1-r)$-separated family of cardinality $2^{\mathfrak c}$ in a space, $\ell_\infty/c_0$, of density character $\mathfrak c$, which is impossible. Consequently   $J(\ell_\infty/c_0)=2$.
\item The case of ultrapowers is quite similar to the previous one.\end{itemize}

\subsection{Separable injectivity and the Gr\"unbaum expansion constant} Gr\"unbaum considers in \cite{grun1,grun2} the expansion constant $E(X)$ of a Banach space $X$ as the infimum of those $\lambda>0$
such that whenever one has a family $B(X_i, r_i)$ of mutually intersecting balls then $\cap_i B(x_i, \lambda r_i)\neq \emptyset$.
While \cite{grun1} is devoted to the study of the expansion constant in finite-dimensional spaces, \cite{grun2} considers the expansion constant in infinite dimensional spaces and establishes that the extension constant coincides with the projection constant defined as
$$p(X) = \sup_Y \inf_P \{ \|P\|: P: Y\to X\}$$
where $Y$ runs over all $1$-codimensional superspaces of $X$ and $P$ over all projections of $Y$ onto $X$. It is therefore clear that $E(X)=1$ and the infimum is attained (what Gr\"unbaum calls ``$E(X)=1$ is exact") if and only if $X$ is $1$-injective. Since one has \cite[(iii)]{grun1} that $J(X)\leq E(X)$ it turns out that $E(X)=1$ implies $J(X)=1$ so $X$ is $1$-injective and thus $E(X)=1$ is exact. This is relevant since Gr\"unbaum provides in \cite{grun2} an example of a space $X$ for which $E(X)=2=J(X)$ and such that $J(X)$ is exact but $E(X)$ is not.

If one defines the separable extension constant $E_s(X)$ (and its corresponding cardinal versions $E_\aleph(X)$) by simply restricting the size of the family of balls, one still has $J_s(X)\leq E_s(X)$ (and, in general, $E_\aleph(X)\leq J_\aleph(X)$) and therefore, after Theorem \ref{jung} one gets: \emph{ $X$ is $1$-separably injective if and only if $E_s(X)=1$ is exact}, which is  \cite[Lemma 2.30]{accgmLN}. In general, when $X$ is $\lambda$-separably injective then $E_s(X)\leq \lambda$ \cite[Lemma 2.33]{accgmLN}. It is also clear that there is no simple version of the projection constant $p$ that characterizes $J_s$.

\subsection{On the stability of the Jung constants} Most of the known results about stability of the Jung constant can be found in \cite{amir} and most of them treat the case of reflexive spaces and thus are not interesting for us. It was already shown by Amir \cite[Prop.1.1]{amir} that every finite-codimensional subspace of a $C(K)$-space with $K$ not extremally disconnected (i.e., every non $1$-injective $C(K)$-space) has Jung constant $2$, which proves Franchetti's conjecture \cite{fran1} that $J(C(K))<2$ implies $J(C(K))=1$. A version of Amir's theorem which somehow extends Franchetti's conjecture to $J_s$ appears in \cite[Proposition 2.4]{accgmLN}. Recall from \cite{accgmLN} that a $C(K)$ space is $1$-separably injective if and only if $K$ is an $F$-space (cf. \cite{accgmLN} for its definition). Amir proves that $E_s(C(K))<2$ implies that $K$ is an $F$-space, therefore $C(K)$ is $1$-separably injective, and thus $J_s(C(K))=1$. Cardinal modifications of \cite[Proposition 2.4]{accgmLN} obtained taking \cite[Proposition 5.12]{accgmLN} into account yield that $E_\aleph(C(K))<2$ implies that $K$ is an $F_\aleph$-space \cite[Definition 5.19]{accgmLN}, and therefore $C(K)$ is $(1, \aleph)$-separably injective
\cite[Theorem 5.16]{accgmLN}, and thus $J_\aleph(C(K))=1$. This suggests the question of whether $J_\aleph(\mathcal C)<2$ implies $J_\aleph(\mathcal C)=1$ for a space $C(K)$ of continuous functions on a compact space $K$.

Given a Banach space parameter, say $\eta(\cdot)$, it is usually a relevant question whether $\eta(X)=\eta(X^{**})$.   For instance, the problem of whether $T(X)=T(X^{**})$ for the thickness or Whitley constant \cite{whit} was posed in \cite{castpapisimo} and negatively solved in \cite{aaa}, while the same problem for the Kottman constant was posed in \cite{castpapisimo} and negatively solved in \cite{cgp}. The situation for Jung constants is not different, although in this case we already know that $J(X)$ (resp. $J_s(X)$) and $J(X^{**})$ (resp. $J_s(X^{**})$) can be different. Indeed, $J(c_0)=2=J_s(c_0)$ while $J(\ell_\infty)=J_s(\ell_\infty)=1$.
In general, given a countably incomplete ultrafilter $\U$ on $\N$, one has $J_s(C[0,1])=2$ but $J_s(C[0,1]_\U)=1$ since according to \cite{accgm4,accgmLN} the ultrapower of a Lindenstrauss space is $1$-separably injective, as we have already mentioned. On the other hand, $J(\ell_\infty)=1$ but $J((\ell_\infty)_\U)=2$ since, again according to \cite{accgm4,accgmLN} no infinite-dimensional ultrapower is injective.

\subsection{$1$-separably injective spaces} In \cite[Problem 7]{accgmLN} it is posed the problem of the existence of $1$-separably injective subspaces of $\ell_\infty$ not isomorphic to $\ell_\infty$. We can present a partial answer:

\begin{lemma}\label{piero} A $1$-separably injective subspace $\Theta \subset \ell_\infty$ containing the canonical copy of $c_0$ is $\ell_\infty$.
\end{lemma}
\begin{proof} Every point $x=(x_n) \in \ell_\infty$ is the unique center of a countable set of $c_0$: the set $\{(x_n +\|x\|)e_n, (x_n -\|x\|)e_n: n\in \N\}$ has diameter $2\|x\|$ and radius $\|x\|$ with $x$ the only center.\end{proof}

The result can be extended to subspaces $\Theta$ whose unit ball is weak*-dense in the unit ball of $\ell_\infty$, but one is still far from
showing even that no $1$-separably injective subspace of $\ell_\infty$ containing an isometric copy of $c_0$, different from $\ell_\infty$, exists. The following application was suggested by F. Cabello:

\begin{corollary} $\beta \N$ is the only compactification of $\N$ that is an $F$-space
 \end{corollary}
 \begin{proof} Let $\gamma \N$ be a compactification of $\N$. An extension of the canonical map $\N\to \gamma\N$ yields a quotient map $\beta\N \to \gamma \N$. Thus $C(\gamma \N)$ is a subspace of $C(\beta \N)=\ell_\infty$ that contains the canonical copy of $c_0$. If $\gamma \N$ is an $F$-space then $C(\gamma \N)$ is $1$-separably injective, and thus it must be isomorphic to $\ell_\infty$.\end{proof}

Let now $X$ be a subspace of $\ell_\infty$. Consider the set $\mathcal C_0\subset X^\N$ of all countable families of elements of $X$ such that $\|x_n - x_m\|\leq 2$ for $n\neq m$; any family $F\in \mathcal C_0$ has (at least) a center $z(F)\subset \ell_\infty$. Let $z(X)$ be a set containing a center for each family in $\mathcal C_0$. Set $X_1= [X + z(X)]\subset \ell_\infty$ and proceed inductively. Given a countable ordinal $\alpha$ for which $X_\alpha \subset \ell_\infty$ has been chosen, if $\alpha$ is not a limit ordinal then let $\mathcal C_\alpha$ be the set of all countable families of elements of $X_\alpha$ such that $\|x_n - x_m\|\leq 2$ for $n\neq m$. For each $F\in \mathcal C_\alpha$ let $z(X_\alpha)$ be a set of centers of the families in $\mathcal C_\alpha$. Finally set $X_{\alpha+1} = [ X_\alpha + z(X_\alpha)]$ in $\ell_\infty$. If $\alpha$ is a limit ordinal work with the closure of $\bigcup_{\beta<\alpha} X_\beta$. Iterate the process $\omega_1$ steps to get a $1$-separably injective subspace $X_{\omega_1}$ of $\ell_\infty$: indeed, giving a mutually intersecting sequence of balls $\{B(x_n, 1): n\in \N\}$ the set $\{x_n: n\in \N\}$ is contained in some $X_\alpha$, hence it has a center in $X_{\alpha+1}$ and thus in $X_{\omega_1}$. Thus, $J_s(X_{\omega_1})=1$ and, by Theorem \ref{jung}, $X_{\omega_1}$ is $1$-separably injective. There are however many possible choices for the successive sets of centers, so the final resulting space $X_{\omega_1}$ is, in principle, not unique. In fact, if \cite[Problem 7]{accgmLN} above has a negative solution and no $1$-separably injective subspaces of $\ell_\infty$ different from $\ell_\infty$ exist, all spaces $X_{\omega_1}$ must be the isomorphic to $\ell_\infty$. Let us now exam what occurs when the centers are chosen outside $\ell_\infty$. Recall that given a Banach space $X$ there are several constructions in the literature \cite{accgm1,accgm2,accgmLN,castsua} of $1$-separably injective spaces $\mathscr S(X)$ containing an isometric copy of $X$ and with the additional property that every operator $\tau: X\to \mathcal S$ from $X$ into a $1$-separably injective space  $\mathcal S$ admits an equal norm extension $\mathscr S(X)\to \mathcal S$. These were called $1$-separably injective covers of $X$. One has:

\begin{lemma} $\;$\begin{enumerate}
\item Under {\sf CH}, there is a subspace of $G\subset \ell_\infty$ such that no $1$-separably injective cover $\mathscr S(G) $ can be isomorphic to $\ell_\infty$ or a subspace of $\ell_\infty$.
\item Under {\sf MA +} $\mathfrak c = \aleph_2$, any $1$-separably injective cover $\mathscr S(c_0) $ of $c_0$ is not isomorphic to $\ell_\infty$ and is not a subspace of $\ell_\infty$ .
\end{enumerate}
\end{lemma}
 \begin{proof} We prove first assertion (1). Consider any embedding $\mathcal G\to C[0,1]$ where $\mathcal G$ denotes Gurariy space and form the resulting exact sequence
$$\begin{CD}
0 @>>> \mathcal G @>>> C[0,1] @>>> Q @>>> 0
\end{CD}
$$
No ultrapower of this sequence splits since, by the results in \cite{accgmc,accgmLN}, no ultrapower of the Gurariy space is a complemented subspace of any $C(K)$-space. Write $Q_\U$ as a quotient of $\ell_1(\mathfrak c)$ with kernel $K$ and form the commutative diagram
\begin{equation}\label{gu}
\begin{CD}0@>>>K @>>>\ell_1(\mathfrak c) @>>> Q_\U@>>>0\\
&& @V{\tau}VV @VVV @|\\
0@>>>\mathcal G_\U@>>> C[0,1]_\U @>>> Q_\U@>>>0
\end{CD}
\end{equation}
in which the operator $\tau$ cannot be extended to an operator $\ell_1(\mathfrak c)\to \mathcal G_\U$. Since $\ell_1(\mathfrak c)$ is a subspace of $\ell_\infty$, $\tau$ cannot be extended to $\ell_\infty$. Now, observe that if a subspace $X$ of $\ell_\infty$ admits an embedding $i: X\to \ell_\infty$ such that operators from $X$ into a space $E$ can be extended to $\ell_\infty$ through $i$ then every embedding $j: X\to \ell_\infty$ enjoys the same property as $i$. Thus, a subspace $X$ of $\ell_\infty$ such that some $\mathscr S(X)$ is isomorphic to $\ell_\infty$ admits an embedding $X\to \ell_\infty$ with the property that every operator $X\to E$ from $X$ into any $1$-separably injective space can be extended to $\ell_\infty$. And consequently, the same occurs to any embedding $X\to \ell_\infty$. In conclusion, pick $X=K$ to get that  $\mathscr S(K)$ cannot be isomorphic to $\ell_\infty$.

Pick now $G=c_0 + K$. Any $1$-separably injective subspace of $\ell_\infty$ containing $G$ must therefore be $\ell_\infty$ by Lemma \ref{piero}. On the other hand, since $K$ is Schur, $c_0$ and $K$ are incomparable and thus $c_0 + K= c_0\oplus K$. There is therefore an operator $c_0 + K\to \mathcal G_\U$ that cannot be extended to $\ell_\infty$ and thus no $1$-separably injective cover of $G$ can be a subspace of $\ell_\infty$.\medskip

To prove (2) we need to consider the $1$-separably injective space $\mathscr{AK}$ that does not contain $\ell_\infty$ constructed by Avil\'es and Koszmider \cite{ak} under the axioms {\sf MA +} $\mathfrak c = \aleph_2$. Since $\mathscr{AK}$ is separably injective, it contains $c_0$ \cite{accgm1,accgmLN} but the inclusion $c_0\to \mathscr {AK}$ cannot be extended  to an operator $\ell_\infty\to \mathscr{AK}$: indeed, any such extension should be, by Rosenthal's theorem \cite{rose}, either a weakly compact operator, which is impossible, or an isomorphism on a copy of $\ell_\infty$, which is impossible too since $\mathscr{AK}$ does not contain $\ell_\infty$. Consequently, $\mathscr S(c_0) \neq \ell_\infty$. By
Lemma \ref{piero}, $\mathscr S(c_0)$ cannot be a subspace of $\ell_\infty$.\end{proof}

The difference between the two results is that the subspace $G$ of (1) cannot be separable since, under {\sf CH}, $\ell_\infty$ is a $1$-separable injective cover of every separable space. A related topic is the open problem of whether $1$-separably injective spaces with density character at most $\mathfrak c$ must be quotients of $\ell_\infty$. Observe that under
${\sf MA} + \mathfrak c = \aleph_2$ the space $\mathscr{AK}$ cannot be a quotient of $\ell_\infty$.

\section{The interplay between the Jung and Kottman constants:\\
 extension of Lipschitz maps}

If $X$ is an infinite-dimensional Banach space with unit ball $B(X)$, the Kottman constant \cite{Ko} of $X$ is defined as
$$K(X) =\sup_{(x_n)\in B(X)} \sep(x_n)$$
where, for a given sequence $(x_n)$, we define $ \sep((x_n)_n)=\inf_{n\neq m} \|x_n-x_m\|.$ A well known result of Elton and Odell \cite{EO} asserts that $K(X)>1$ for every infinite-dimensional Banach space $X$.
The finite Kottman constant is defined as
$$ K_f(X)=\; \sup \{r>0: \forall n \in \N \quad \exists\; A: |A|=n\; \mathrm{and}\; r-\mathrm{separated}\}$$

We define the Lipschitz expansion constants: $e(X, Z)$ (resp. $e_1(X, Z)$) is the infimum of all $\lambda>0$ such that for every subset $M$ of $X$ (resp. every subset $M$ and every point $x\in X$) every Lipschitz map $f:M\to Z$ admits a Lipschitz extension $F: X\to Z$ (resp. $F: M\cup{\{x\}}\to Z$) with $\Lip(F)\leq \lambda \Lip(f)$. It will be useful for us to define the constants $e_1^s(X, Z)$ (resp. $e_1^f(X, Z))$ as the versions of $e_1(X, Z)$ obtained allowing only separable (resp. finite) $M$. It is clear that $e_1^f \leq e_1^s \leq e_1$. One has

\begin{prop}\label{kottmanjung} For every couple $Y,X$ of infinite dimensional Banach spaces one has
\begin{enumerate}
\item\label{kje} $K(Y)J_s(X) \leq 2e_1^s(Y,X).$
\item\label{kjef}$K_f(Y)J_f(X) \leq 2e_1^f(Y,X).$
\item $K_f(Y)J_s(X) \leq 2e_1^s(Y,X^{**}).$
\end{enumerate}
The inequalities above are sharp and $J_s$ cannot be replaced by $J$; precisely, $K(Y)J(X)\leq 2e_1(Y, X)$ does not hold.
\end{prop}

\begin{proof} Let $\varepsilon>0$. Consider in $X$ a bounded countable set $A$ such that
$\frac{2 \, r(A)}{\delta(A)}> J_s(X) - \varepsilon$ and $\delta(A)=K(X)-\varepsilon$. Choose in the unit ball of $Y$ an infinite, countable set $C= \{y_1, . . . , y_n, . . . \}$ such that
$K(Y)-\varepsilon \leq||y_i - y_j|| \leq K(Y)+\varepsilon$ for every pair $ i, j \in N;\;   i \neq j$ (see \cite [3.4 Lemma]{DB}).  Consider any bijective map $ \tau: C \rightarrow A$. This map is $1$-Lipschitz since
$$\|\tau y_i-\tau y_j\|\leq \delta (A) \leq \|y_i-y_j\|. $$

Let $T: C \cup  \{0\} \to X$ be a Lipschitz extension of $\tau$ with Lipschitz constant not greater than
$e_1^s(Y, X)+\varepsilon$. It turns out that

\begin{eqnarray*} \frac{1}{2} (J_s(X) - \varepsilon) (K(Y)-\varepsilon) &=&  \frac{1}{2} (J_s(X) - \varepsilon) \delta(A)\\  &<& r(A)\\ &\leq&  \sup \{||T(0)- \tau y_i||: i \in N \} \leq e_1^s(Y, X)+\varepsilon.\end{eqnarray*}
Letting $\varepsilon \to 0$ proves part (1) of the lemma.

 To obtain the finite version (2) we need a result which is likely to be known, but for which we could only find an abstract formulation in \cite[Cor. 1.4]{kania} , so we include a simple proof here for the sake of completeness.

\begin{lemma} Let $X$ be a Banach space and let $\varepsilon>0$. For every  $n\in \N$ it is possible to choose in the unit ball of $X$ a set $\{x_1, . . . , x_n\}$ such that
$K_f(X)-\varepsilon \leq||x_i-x_j|| \leq K_f(X)+\varepsilon$ for every pair $ i, j = 1, . . . ,  n ;\;   i \neq  j$.
\end{lemma}
\begin{proof} For some $n_1 \in \N$ one cannot find $n_1$ points in $B(X)$ which are $                                                                (K_f(X)+\varepsilon)$-separated. According to the finite Ramsey theorem, there is a number $R(n,n_1)$  such that whenever one has $n_2 \geq R(n_1)$ points which are $(K_f(X)-\varepsilon)$-separated, there must be either $n$ points whose mutual distance is between $K_f(X)-\varepsilon$ and $K_f(X) + \varepsilon$, or $n_1$ points $K_f(X)+\varepsilon$-separated, which is impossible. So only the first case holds, which proves the assertion.\end{proof}

The rest of the argument is as before.

Regarding the proof of (3), it was observed in \cite{castpapi} that $K_f(Y)=K(Y_{\mathscr U})$ for every countably incomplete ultrafilter $\mathscr U$. Therefore, inequality (1) becomes

$$\frac{1}{2}K_f(Y)J_s(X) \leq e_1^s(Y_{\mathscr U}, X). $$

One has:\medskip

\noindent \textbf{Claim.} $e_1^s(Y_{\mathscr U}, X) \leq  e_1^s(Y, X^{**})$.\medskip

\noindent \emph{Proof of the Claim.} Let $M\subset Y_\U$
be a countable subset, $x\notin M$ and $f: M\to X$ a Lipschitz map. Assume that $M = \{[m^k_n] : k\in \N\}$ and $x=[x_n]$.  Form the countable subsets $M_k = \{m^k_n: n\in \N\}\subset Y$ and consider the restriction $f_k$ of $f$ to the set $M_k$ inside the natural (diagonal) copy of $Y$ inside $Y_\U$. Let $F_k$ be an extension to $M_k \cup \{x_k\}$ with Lipschitz constant $e_1^s(Y, X) + 1/k$. Form the element
$[F_k]: M\cup \{x\} \to X_\U$ given by $[F_k][z_n] = [F_k(z_k)]$, which is a Lipschitz map with Lipschitz constant $e_1^s(Y, X)$. Pick a norm one projection $X_\U\to X^{**}$ to get a Lipschitz map $P[F_k]: M\cup\{x\}\to X^{**}$ with Lipschitz constant $e_1^s(Y, X)$.\hfill $\square$\medskip

One thus gets $\frac{1}{2}K_f(Y)J_s(X) \leq e_1^s(Y, X^{**})$.\medskip

None of the second group of inequalities can hold: pick as $X$ any $1$-separably injective space that is not $1$-injective and  $Y$ separable such that $K(Y)=2$. The inequality $K(Y)J(X)\leq e_1^s(Y, X)$ forces $J(X)=1$ which cannot be. The inequality $K(Y)J(X)\leq e_1(Y,X)$ also fails: pick $Y=c_0$ and $X=\ell_\infty/c_0$; then $K(c_0) = J(\ell_\infty/c_0) =2$  and $e_1(c_0, \ell_\infty/c_0) = e_1^s(c_0, \ell_\infty/c_0) =1$.
\end{proof}

\section{Applications.}

\subsection{} Kalton proved in \cite[Proposition 5.8]{Kal} the unexpected result that
$$K(X)\leq e(X,c_0)$$
for every infinite dimensional Banach space $X$. See \cite{jesus} for an  extension to $\alpha$-H\"older maps. We derive now from Proposition \ref{kottmanjung} a simple proof for an improved version.
\begin{prop} For every infinite dimensional Banach space $X$ one has .
$$K(X)\leq e_1^s(X,c_0)$$
\end{prop}
\begin{proof} The result follows from $J_s(c_0)=2$ and $K(X)J_s(c_0) \leq 2e_1^s(X,c_0)$.
\end{proof}
 \subsection{} Kalton shows in \cite{Kal} that if $e_1(Y, X)=\lambda$ then norm one operators from subspaces of $Y$ to $X$ admit, for each $\varepsilon>0$, an  extension to  one more dimension with norm at most $\lambda+\varepsilon$. The argument is simple: pick $A$ a subspace of $Y$ and $b\notin A$. If a norm one operator $\tau: A\to X$ extends to a Lipschitz map $L: A \cup {\{b\}}\to X$ with Lipschitz constant $\lambda$, set $T: A + [b]\to X$ the operator
$T(b)=L(b)$. Then for every $a\in A$ one has $\|T(a-b)\|=\|La - Lb\|\leq \lambda \|a-b\|$, from where $\|T\|\leq \lambda$. One thus has:

\begin{lemma}\label{klemma} A Banach space $X$ such that $e_1(Y, X)=1$ for every Banach space $Y$ must be $1$-injective.\end{lemma}

Let us show how this result can be completed. Recall that the argumentation above cannot be used for $1$-separable injectivity. One has:

\begin{teor} $\;$
\begin{itemize}
\item A Banach space $X$ is $1$-injective if and only if $e_1(Y, X)=1$ for all Banach spaces $Y$.
\item A Banach space $X$ is $1$-separably injective if and only if $e_1^s(Y, X)=1$ for all Banach spaces $Y$.
\end{itemize}
\end{teor}
\begin{proof} To prove the first assertion we just need to obtain the converse of Kalton's lemma \ref{klemma}, and this is consequence of

\begin{lemma} If all operators $E\to X$ extend with the same norm to operators $Y\to X$ whenever $\dim Y/E=1$ then Lipschitz maps $M\to X$ extend to Lipschitz maps $M\cup \{p\}\to X$ with the same Lipschitz constant.\end{lemma}
\begin{proof} Let $\textbf{Lip}(M, \R)_0$ be the space of Lipschitz maps $\ell: M\to \R$ such that $\ell(m_0)=0$ for a fixed $m_0\in M$. It is a Banach space endowed with the natural norm $\|\ell\|= Lip(\ell)$ where $Lip$ denotes the Lipschitz constant of $\ell$. Let $\Delta: M \to \textbf{Lip}(M, \R)_0^*$ be the canonical Lipschitz map given by $\Delta_m(\ell) = \ell(m)$. The Lipschitz-free space \cite{godeka} $\mathcal F(M)$ is the closure of the image $\delta(M)$. It has the property that for every Lipschitz map $L: M\to X$ such that $L(m_0)=0$ there is an operator $\phi_L: \mathcal F(M)\to X$ such that $\phi_L\Delta = L$ and $\|\phi_L\|= Lip(L)$.

Now, let $L:M\to X$ be a Lipschitz map. With a translation there is no loss of generality assuming that $L(m_0)=0$. Since
 $\mathcal F(M)$ is a one-codimensional subspace of $\mathcal F(M\cup p)$, the operator $\phi_L: \mathcal F(M) \to X$ extends with the same norm to an operator $\widehat{\phi_L}: \mathcal F(M\cup p)\to X$, and thus $\widehat{\phi_L}\Delta$ is a Lipschitz map extending $L$ with the same Lipschitz constant.\end{proof}

 To prove the second equivalence we use the inequality $K(Y)J_s(X)\leq e_1^s(Y,X)$ to get that a Banach space $X$ such that $e_1^s(Y,X)=1$ for some Banach space $Y$ with $K(Y)=2$ must have $J_s(X)=1$ and thus must be $1$-separably injective. Conversely, if a Banach space is $1$-separably injective then $e_1(Y, X)=1$ for every separable Banach space $Y$. Now, if $e_1(Y, X)=1$ for all separable Banach spaces $Y$ then $e_1^s(Y, X)=1$ for all Banach spaces $Y$.\end{proof}

\end{document}